\documentclass[12pt]{article}
\input epsf
\usepackage{amsmath}

\usepackage{amssymb}
\usepackage{theorem}

\usepackage{enumerate}

\usepackage{mathrsfs}

\usepackage{geometry}

\numberwithin{equation}{section}

\newtheorem{Theorem}{Theorem}[section]
\newtheorem{Lemma}[Theorem]{Lemma}

\newtheorem{cor}[Theorem]{Corollary}
{\theorembodyfont{\rmfamily}

\newtheorem{Remark}[Theorem]{Remark}
}

\newcommand{\qed}{\hfill \mbox{\raggedright \rule{.07in}{.1in}}}

\newenvironment{proof}{\vspace{1ex}\noindent{\bf
Proof}\hspace{0.5em}}{\hfill\qed\vspace{1ex}}

\newcommand{\T}{{\mathbb T}}

\def\P{\ensuremath{\mathcal P}}

\def\F{\ensuremath{\mathcal F}}

\def\TF{\mathcal{T}}

\def\eps {\varepsilon}

\begin{document}
\bibliographystyle{plain}

\title {Almost sure invariance principle for sequential and non-stationary dynamical systems}

\author{Nicolai Haydn\thanks {University of Southern California, Los Angeles.
e-mail:$<$nhaysdn@math.usc.edu$>$.}
\and Matthew Nicol \thanks{Department of Mathematics,
University of Houston,
Houston Texas,
USA. e-mail: $<$nicol@math.uh.edu$>$.}
\and Andrew T\"or\"ok \thanks{Department of Mathematics,
University of Houston,
Houston Texas,
USA. e-mail: $<$torok@math.uh.edu$>$.}
\and Sandro Vaienti
\thanks{Aix Marseille Universit\'e, CNRS, CPT, UMR 7332, 13288 Marseille, France and
Universit\'e de Toulon, CNRS, CPT, UMR 7332, 83957 La Garde, France.
e-mail:$<$vaienti@cpt.univ-mrs.fr$>$.}}

\maketitle
\tableofcontents{}

\begin{abstract}
We establish almost sure invariance principles, a strong form of approximation by Brownian motion, for non-stationary time-series arising as observations on dynamical systems.  Our
examples include observations  on sequential expanding maps, perturbed dynamical systems, non-stationary sequences of functions on hyperbolic systems as well as applications to the shrinking target 
problem in expanding systems.

\end{abstract}

\section{Introduction}
\setcounter{equation}{0}

A recent breakthrough by Cuny and Merlev\`ede~\cite{Cuny_Merlevede}  establishes conditions under which  the almost sure invariance principle (ASIP) holds for reverse martingales.
The ASIP is a matching of  the trajectories of the dynamical system with a Brownian motion in such a way that the error is negligible in comparison with the Birkhoff sum.  Limit theorems such
as the central limit theorem, the functional central limit theorem and the  law of the iterated logarithm transfer from the Brownian motion to time-series generated by observations on the  dynamical system. 

Suppose $\{U_j\}$ is a sequence of random variables on a probability space $(X,\mu)$ with $\mu (U_j)=0$ for all $j$. Define $\sigma_n^2= \int (\sum_{j=1}^n U_j)^2d\mu$ and suppose that 
$\lim_{n\to \infty} \sigma_n^2=\infty$. We will say $(U_j)$ satisfies the ASIP if  there is a sequence of independent centered Gaussian random variables $(Z_j)$ such that, enlarging our probability space if necessary,
\[
\sum_{j=1}^n U_j =\sum_{j=1}^n Z_j + O (\sigma_n^{1-\gamma})
\]
almost surely for some $\gamma>0$ and furthermore
\[
\sum_{j=1}^n E[Z_j^2]= \sigma_n^2 + O(\sigma_n^{(1+\eta)} )
\]
for some $0<\eta<1$. 

If $(U_j)$ satisfies the ASIP then $(U_j)$ satisfies the (self-norming) CLT and 
\[
\frac{1}{\sigma_n} \sum_{j=1}^{n} U_j \to \mbox{N}(0,1)
\]
where  the convergence is in distribution.

Furthermore if $(U_j)$ satisfies the ASIP then $(U_j)$ satisfies the law of the iterated logarithm and 
\[
\limsup_n [\sum_{j=1}^n U_j ]/\sqrt{\sigma_n\log \log (\sigma_n)} =1
\]
while 
\[
\liminf_n [\sum_{j=1}^n U_j ]/\sqrt{\sigma_n\log \log (\sigma_n)} =-1
\]
In fact there is a matching of the Birkhoff sum $\sum_{j=1}^n U_j$ with a standard Brownian motion $B(t)$ observed at times $t_n=\sigma_n^2$ so that $\sum_{j=1}^n U_j=B(t_n)$ (plus error) almost surely.

In the Gordin~\cite{Gordin} approach
to establishing the central limit theorem (CLT), reverse martingale difference schemes arise naturally.  To establish distributional limit theorems for stationary dynamical systems, such as the  central limit theorem, it is possible to reverse time and use the martingale central limit theorem in backwards time to establish the CLT for the original system. This approach does not a priori work for the almost sure invariance principle, nor for  other almost sure limit theorems. To circumvent this problem Melbourne and Nicol~\cite{MN1,MN2} used  results of Philipp and Stout~\cite{Philipp_Stout} based upon
the Skorokhod embedding theorem to establish the ASIP for H\"older functions on a class of non-uniformly hyperbolic systems, for example those
modeled by Young Towers. Gou\"ezel~\cite{Gouezel} used spectral methods to give error rates in the ASIP for a wide class of dynamical systems, and his formulation does not require the assumption
of a Young Tower. Rio and Merlev\`ede~\cite{Merlevede_Rio} established the ASIP for a broader class of observations, satisfying only mild integrability conditions, on piecewise
expanding maps  of $[0,1]$.

We will need the following theorem of Cuny and Merlev\`ede:

\begin{Theorem}\cite[Theorem 2.3]{Cuny_Merlevede}\label{Cuny_Merlevede}
Let $(X_n)$ be a sequence of square integrable random variables adapted to a non-increasing filtration $(\mathcal{G}_n)_{n\in N}$. Assume that $E(X_n|\mathcal{G}_{n+1})=0$
a.s., that $\sigma_n^2:=\sum_{k=1}^n E(X_k^2)\to \infty$ and that $\sup_n E(X_n^2) <\infty$. Let $(a_n)_{n\in N}$ be a non-decreasing sequence of positive numbers such that
$(a_n/\sigma_n^2)_{n\in N}$ is non-increasing and $(a_n/\sigma_n)_{n\in N}$ is non-decreasing. Assume that
\begin{eqnarray*}
&\textbf{(A)\qquad}~&\sum_{k=1}^n (E(X_k^2|\mathcal{G}_{k+1})-E(X_k^2))=o(a_n)\qquad P-a.s.\\
&\textbf{(B)\qquad}~&\sum_{n\ge 1}a_n^{-v}E(|X_n|^{2v})<\infty~\mbox{ for some $1\le v\le 2$}
\end{eqnarray*}
Then enlarging our probability space if necessary it is possible to find a sequence $(Z_k)_{k\ge 1}$ of independent centered Gaussian
variables with $E(Z_k^2)=E(X_k^2)$ such that
\[
\sup_{1\le k\le n}\left|\sum_{i=1}^k X_i -\sum_{i=1}^k Z_i\right|
=o\!\left((a_n(|\log(\sigma^2_n/a_n)|+\log\log a_n))^{1/2}\right)\qquad P-a.s.
\]
\end{Theorem}

We use this result to  provide sufficient
conditions  to obtain  the  ASIP for H\"older or BV observations
on a large class of expanding sequential dynamical systems.  We also obtain
the ASIP for some other classes of non-stationary dynamical systems,
including ASIP limit laws for  the shrinking target problem on  a class of expanding maps and 
non-stationary observations on Axiom A dynamical systems.

  In some of our examples the variance  $\sigma_n^2$ grows linearly  $\sigma_n^2\ \sim n \sigma^2$ so that 
$S_n =\sum_{j=1}^n \phi_j \circ T^j$ is  approximated by $ \sum_{j=1}^n Z_j (=B(\sigma^2 n)) $
where $Z_j$ are iid Gaussian all with variance $ \sigma^2$ and $B(t)$ is standard
Brownian motion. We will call this case a standard ASIP with variance $\sigma^2$.

In other settings, like the shrinking target problem, $\sigma_n^2$ does not
grow linearly. In fact we don't know precisely its rate of increase,  just that it goes to infinity.
In these cases $S_n=\sum_{j=1}^n U_j $ is approximated by $\sum_{j=1}^nZ_j=B(\sigma_n^2)$  where the $Z_j$ are
independent Gaussian but not with same variance, in fact
$Z_j =B(\sigma_{j+1}^2)- B (\sigma_{j}^2)$ is a Brownian motion increment,
the time difference (equivalently variance) of which varies with $j$.

Part of the motivation for this work is to extend our statistical understanding of physical processes from the 
stationary to the non-stationary setting, in order to better model non-equilibrium or time-varying systems. Non-equilibrium statistical physics is
a very active field of research but ergodic theorists have until recently focused on the stationary setting. 
The notion of loss of memory for non-equilibrium dynamical systems was introduced and studied in the  work of  Ott, Stenlund and Young \cite{OSY}, but this notion only concerns the rate
of convergence of initial distributions  (in a metric on the space of measures) under the time-evolution afforded by the dynamics. In this paper we consider more refined 
statistics on a variety of non-stationary dynamical systems.

The term {\it sequential dynamical systems}, introduced by Berend and
Bergelson~\cite{Berend_Bergelson}, refers to a (non-stationary) system in
which a sequence of concatenation of maps $T_k\circ T_{k-1} \circ \ldots
\circ T_1$ acts on a space, where the maps $T_i$ are allowed to vary
with~$i$. The seminal paper by Conze and Raugi \cite{CR} considers the CLT
and dynamical Borel-Cantelli lemmas for such systems. Our work is based to
a large extent upon their work. In fact we show that the (non-stationary)
ASIP holds under the same conditions as stated in~\cite[Theorem 5.1]{CR} (which implies
the non-stationary CLT), provided a mild condition on the growth of the
variance is satisfied.

We consider families ${\cal F}$ of non-invertible maps $T_{\alpha}$ defined
on compact subsets $X$ of $\mathbb{R}^d$ or on the torus $\mathbb{T}^d$
(still denoted with $X$ in the following),  and non-singular with respect
to the Lebesgue or the Haar measure
i.e. $m(A) \not= 0 \implies m(T(A)) \not= 0$.
Such measures will be defined on the Borel sigma algebra ${\cal B}$. We will be mostly concerned with the case $d=1.$  We fix a family ${\cal F}$ and take a countable sequence of maps $\{T_k\}_{k\ge 1}$ from it: this sequence defines a {\em sequential dynamical system}. A {\em sequential orbit} will be defined by the concatenation
\begin{equation}\label{m}
\TF_n:=T_n\circ\cdots\circ T_1, \ n\ge 1
\end{equation}
We  denote with $P_{\alpha}$ the Perron-Frobenius (transfer) operator associated to $T_{\alpha}$ defined by the duality relation
$$
\int_M P_{\alpha}f \ g\ dm\ = \ \int_M f\ g\circ T_{\alpha} \ dm,\;\; \mbox{ for all } f\in \mathscr{L}^1_m, \ g\in \mathscr{L}^{\infty}_m
$$
Note that here the transfer operator $P_{\alpha}$ is defined with respect to the reference measure $m$, in later sections we will consider the transfer operator 
defined by duality with respect to a natural invariant measure. 

Similarly to (\ref{m}), we define the composition of operators as
\begin{equation}\label{o}
\P_n:=P_n\circ\cdots\circ P_1, \ n\ge 1
\end{equation}
It is easy to check that duality persists under concatenation, namely
\begin{equation}\label{c}
\int_M g  (\TF_n) \  f \ dm=\int_M g (T_n\circ\cdots\circ T_1) \ f\ dm\ =\ \int_M g(\ P_n\circ\cdots\circ P_1f )\ dm\ =  \int_M g\ (\P_n f) \ dm
\end{equation}
To deal with probabilistic features of these systems, the martingale approach is  fruitful. We now introduce the basic concepts and notations.

We define ${\cal B}_n:=\TF_n^{-1} {\cal B},$ the $\sigma$-algebra associated to the $n$-fold pull back
 of the Borel $\sigma$-algebra ${\cal B}$ whenever $\{T_k\}$ is a given sequence in the family ${\cal F}.$
  We set ${\cal B}_{\infty}=\bigcap_{n\ge 1} \TF_n^{-1} {\cal B}$ the {\em asymptotic
   $\sigma$-algebra}; we say that the sequence $\{T_k\}$ is {\em exact} if ${\cal B}_{\infty}$ is trivial.
    We take $f$ either in $\mathscr{L}^1_m$ or in $\mathscr{L}^{\infty}_m$ whichever makes sense in the following expressions. It was proven in \cite{CR} that  for $f\in \mathscr{L}^{\infty}_m$  the quotients $|\P_nf/\P_n1|$ are bounded by $\|f\|_{\infty}$ on $\{\P_n1>0\}$ and $\P_nf(x)=0$ on the set $\{\P_n1=0\},$
    which allows us to define  $|\P_nf/\P_n1|=0$ on $\{\P_n1=0\}.$ We therefore have, the expectation
     being taken w.r.t.\ the Lebesgue measure:
\begin{equation}\label{e1}
\mathbb{E}(f|{\cal B}_k)\ = (\frac{\P_kf}{\P_k1})\circ \TF_k
\end{equation}
\begin{equation}\label{e2}
\mathbb{E}(\T_lf|{\cal B}_k)\ = (\frac{P_k\cdots P_{l+1}(f\P_l1)}{\P_k1})\circ \TF_k, \ 0\le l\le k\le n
\end{equation}
Finally the martingale convergence  theorem ensures that for $f\in \mathscr{L}^1_m$ there is
convergence of the conditional expectations $(\mathbb{E}(f|{\cal B}_n))_{n\ge 1}$ to
$\mathbb{E}(f|{\cal B}_{\infty})$ and therefore
$$
\lim_{n\rightarrow \infty} ||(\frac{\P_nf}{\P_n1})\circ \TF_n-\mathbb{E}(f|{\cal B}_{\infty})||_1=0,
$$
the convergence being $m$-a.e.

\section{Background and assumptions.}
In \cite{CR} the authors studied extensively a class of $\beta$ transformations. We consider  a similar  class  of examples and we will also provide  some new  examples for the theory developed in the next section. For each map we will give as well  the properties needed the prove the ASIP; in particular we require  two assumptions which we call, following~\cite{CR}, the~(DFLY) and~(LB) conditions. \\

Property~(DFLY) is a uniform Doeblin-Fortet-Lasota-Yorke inequality for
concatenations of transfer operators; to introduce it we first need to
choose a suitable couple of adapted spaces. Due to the class of maps
considered here, we will consider a Banach space ${\cal V} \subset \mathscr{L}^1_m$
($1\in\mathcal{V}$)
of functions over $X$ with norm $||\cdot||_{\alpha}$, such that
$\|\phi\|_{\infty} \le C \|\phi\|_{\alpha}$.

 For example we could let  ${\cal V} $ be the Banach space  of bounded variation functions over $X$ with norm $||\cdot||_{BV}$ given by the sum of the $\mathscr{L}^1_m$ norm and the total variation $|\cdot|_{bv}.$ or we could take ${\cal V}$ to be the space of Lipschitz or H\"older functions.

\noindent {\bf Property (DFLY):}
Given the family ${\cal F}$  there exist constants $A, B<\infty, \rho\in(0,1),$ such that for any $n$ and any sequence of operators $P_n,\cdots,P_1$ in ${\cal F}$ and any $f\in {\cal V}$ we have
\begin{equation}\label{LY2}
\|P_n\circ \cdots\circ P_1 f\|_{\alpha}\le A \rho^n \|f\|_{\alpha}+B\|f\|_1
\end{equation}

\noindent {\bf Property (LB)}:
 There exists $\delta>0$ such that for any sequence  $P_n,\cdots,P_1$ in ${\cal F}$  we have
 the uniform lower bound
\begin{equation}\label{MIN}
\inf_{x\in M}P_n\circ \cdots\circ P_11(x)\ge \delta,  \quad \forall n\ge 1.
\end{equation}

\section{ASIP for sequential expanding maps of the interval.}

In this section we show that with an additional growth rate condition on the variance the assumptions of~\cite[Theorem 5.1]{CR} imply not just the CLT but the ASIP
as well.

Let $\mathcal{V}$ be a Banach space with norm $\|.\|_{\alpha}$ such that $\|\phi\|_{\infty} \le C \|\phi\|_{\alpha}$. If $(\phi_n)$ is a sequence in $\mathcal{V}$ define $\sigma_n^2=E(\sum_{i=1}^n \tilde{\phi}_i(T_i\cdots T_1 ) )^2$ where
$\tilde{\phi}_n=\phi_n-m(\phi(T_n\cdots T_1))$. 
We write $E[\phi]$ for the expectation of $\phi$ with respect to Lebesgue measure.

\begin{Theorem}\label{Main_BV}
Let $(\phi_n)$ be a sequence in $\mathcal{V}$ such that $\sup_n
\|\phi_n\|_{\alpha} <\infty$ and hence $\sup_n E|\phi_n|^4<\infty$.
Assume~(DFLY) and~(LB) and $\sigma_n \ge n^{1/4+\delta}$ for some $0<\delta
< \frac{1}{4}$. Then $(\phi_n\circ \TF_n)$ satisfies the ASIP i.e.
enlarging our probability space if necessary it is possible to find a sequence $(Z_k)_{k\ge 1}$ of independent centered Gaussian
variables  $Z_k$ such that for any $\beta<\delta$
\[
\sup_{1\le k\le n}|\sum_{i=1}^k \tilde{\phi}_i(T_i\cdots T_1 ) -\sum_{i=1}^k Z_i|=o(\sigma_n^{1-\beta}) \qquad m-a.s.
\]
Furthermore   $ \sum_{j=1}^{n} E[Z_i^2]=\sigma_n^2+O(\sigma_n)$.
\end{Theorem}

\begin{proof}
As above let $\P_n=P_nP_{n-1}\cdots P_1$ and  define as in~\cite{CR} the operators
$Q_n \phi=\frac{P_n(\phi\P_{n-1}1)}{\P_n 1}$. In particular $Q_nT_n\phi=\phi$. With  $h_n$ defined by
$$
h_n=Q_n \tilde{\phi}_{n-1}+Q_nQ_{n-1}\tilde{\phi}_{n-2}+\cdots + Q_nQ_{n-1}\cdots Q_1\tilde{\phi}_0
$$  
we then obtain that
\[
\psi_n=\tilde{\phi}_n + h_n -T_{n+1} h_{n+1}
\]
satisfies $Q_{n+1}\psi_n=0$.
 For convenience let us put $U_n=\TF_n\psi_n$, 
 where, as before, $\TF_n=T_n\circ\cdots\circ T_1$.
As proven by Conze and Raugi~\cite{CR},  $(U_n)$ is a sequence of reversed martingale
differences for the filtration $(\mathcal{B}_n)$.
Note that
\begin{equation}\label{eq.cohomU}
  \sum_{j=1}^n U_j =\sum_{j=1}^n \tilde{\phi}_j (\TF_j) + h_1(\TF_1)- h_n (\TF_{n+1})
\end{equation}
and  $\|h_n\|_{\alpha}$ is uniformly bounded. Hence
  \begin{eqnarray*}
  \left(\sum_{j=1}^n U_j\right)^2&=& \left(\sum_{j=1}^n \tilde{\phi}_j (\TF_j)\right)^2
  +\left(h_1(\TF_1)-h_{n+1}(\TF_{n+1})\right)^2\\
 && +2\left(\sum_{j=1}^n \tilde{\phi}_j (\TF_j)\right)\left(h_1(\TF_1)-h_{n+1}(\TF_{n+1})\right)
  \end{eqnarray*}
and integration yields
  $$
  E\left(\sum_{j=1}^nU_j\right)^2=\sigma_n^2+\mathcal{O}(\sigma_n),
  $$
  where we used that $h_n$ is uniformly bounded in $\mathscr{L}^{\infty}$ (and $\sigma_n\to\infty$).
 Since $\int U_jU_i=0$ if $i\not=j$ one has $\sum_{j=1}^n E(U_j^2)=E\left(\sum_{j=1}^nU_j\right)^2=\sigma_n^2+\mathcal{O}(\sigma_n)$.

 In Theorem~\ref{Cuny_Merlevede}, we will take $a_n$ to be $\sigma_n^{2-\epsilon}$, for some $\epsilon>0$ sufficiently small  ($\epsilon <2\delta$ will do) so that $a_n^2>n^{1/2+\delta^{'}}$
 for all large enough $n$, where $\delta^{'}>0$.
 Then $a_n/\sigma_n^2$ is non-increasing and $a_n/\sigma_n$ is non-decreasing.
 Furthermore Conze and Raugi show that
 $E[U_k^2|\mathcal{B}_{k+1}]=\TF_{k+1}(\frac{P_{k+1}(\psi_k^2\P_k1)}{\P_{k+1}1})$
 and  in~\cite[Theorem 4.1]{CR}  establish that
 \[
 \int [ \sum_{k=1}^n E(U_k^2|\mathcal{B}_{k+1})-E(U_k^2)]^2~dm \le c_1\sum_{k=1}^n E(U_k^2)
 \le c_2\sigma_n^2
 \]
 for some constants $c_1, c_2>0$.  This implies by the Gal-Koksma  theorem (see e.g.~\cite{Sprindzuk}) that
 \[
 \sum_{k=1}^n E(U_k^2|\mathcal{B}_{k+1})-E(U_k^2)=o(\sigma_n^{1+\eta})=o(a_n)
 \]
 $m$ a.s. for any $\eta\in(0,2-\eps)$.  Thus with our choice of $a_n$ we have verified Condition~(A)
 of Theorem~\ref{Cuny_Merlevede}. Taking $v=2$ in Condition~(B) of
 Theorem~\ref{Cuny_Merlevede} one then verifies that $\sum_{n\ge 1} a_n^{-v}E(|U_n|^{2v})<\infty$.

 Thus $U_n$ satisfies the ASIP with error term $o(\sigma_n^{1-\beta})$ for any $\beta<\delta$.
 This concludes the proof, in view of~\eqref{eq.cohomU} and the fact that $\|h_n\|_{\alpha}$
  is uniformly bounded.
 \end{proof}


\section{ASIP for the shrinking target problem: expanding maps.}

We now consider a fixed expanding map $(T,X,\mu)$  acting on the unit interval equipped with a unique ergodic absolutely continuous
invariant probability measure $\mu$. Examples to which our results apply include   $\beta$-transformations,  smooth expanding maps, the Gauss map, and Rychlik maps.
We will define the transfer operator  with respect to the natural invariant measure $\mu$, so that $\int (P f) g\,d\mu=\int fg(T) \,d\mu$ for all $f\in \mathscr{L}^1 (\mu)$,
$g\in \mathscr{L}^{\infty}(\mu)$. 

We assume that the  transfer operator $P$  is quasi compact in the  bounded variation norm so that we have exponential decay of correlations in the bounded variation norm
and $\|P^n \phi\|_{BV} \le C\theta^n \|\phi\|_{BV}$ for all $\phi \in BV(X)$ such that $\int \phi d\mu=0$ (here $C>0$ and $0<\theta <1$ are constants independent of $\phi$).

We say that $(T,X,\mu)$ has exponential decay in  the BV norm versus $\mathscr{L}^1 (\mu)$
if there exist constants $C>0$, $0< \theta <1$ so  that for all
$\phi \in BV$, $\psi \in \mathscr{L}^1(\mu)$ such that $\int \phi \,d\mu=\int \psi\, d\mu=0$:
\[
\left| \int \phi \psi\circ  T^n \,d\mu\right| \le C \theta^n \| \phi\|_{BV} \|\psi \|_{1}
\]
where $\|\psi \|_{1}=\int |\psi|\,d\mu$.
Suppose $\phi_j=1_{A_j}$ are indicator functions of a sequence of nested intervals $A_j$, where $\mu$ is the unique invariant
measure for  the map $T$.

The variance is given by  $\sigma_n^2=\mu(\sum_{i=1}^n \tilde{\phi}_i\circ T^i )^2$, where
$\tilde{\phi}=\phi-\mu (\phi)$ and $E_n=\sum_{j=1}^n \mu (\phi_j)$.

\begin{Theorem}\label{ASIP.sequential}
Suppose $(T,X,\mu)$ is a dynamical system with exponential decay in the BV norm versus
 $\mathscr{L}^1(\mu)$ and whose transfer operator $P$ satisfies $\|P^n \phi\|_{BV} \le C\theta^n \|\phi\|_{BV}$ for all $\phi \in BV(X)$ such that $\int \phi d\mu=0$.
  Suppose $\phi_j=1_{A_j}$ are indicator functions of a sequence
 of nested sets $A_j$ such that
$\sup_n \|\phi_n\|_{BV} <\infty$ and $\frac{C_1}{n^{\gamma}}\le \mu (A_n) $ ($C_1>0$)
where $0< \gamma<1$.
Then $(\phi_n\circ T^n)_{n\ge1}$ satisfies the ASIP i.e.\
enlarging our probability space if necessary it is possible to find a sequence $(Z_k)_{k\ge 1}$
of independent centered Gaussian
variables  $Z_k$ such that for all  $\beta<\frac{1-\gamma}2$
\[
\sup_{1\le k\le n}|\sum_{i=1}^k \tilde\phi_i\circ T^i -\sum_{i=1}^k Z_i|=o(\sigma_n^{1-\beta})\qquad\mu-a.s.
\]

Furthermore $ \sum_{i=1}^n E[Z_i^2]=\sigma_n^2+O(\sigma_n)$.
\end{Theorem}

\begin{proof}
From~\cite[Lemma 2.4]{HNVZ} we see that for sufficiently large $n$, $\sigma_n^2 \ge E_n \ge C n^{1-\gamma}$ for some constant $C>0$
 (note that there is a typo in the statement of ~\cite[Lemma 2.4]{HNVZ}
and $\limsup$ should be replaced with $\liminf$). We follow the proof of Theorem~\ref{Main_BV} based on~\cite[Theorem 5.1]{CR} taking $T_k=T$ for all $k$, $m$ as the invariant measure
$\mu$ and $f_n=1_{A_n}$. Note that conditions (DFLY) and (LB)  are satisfied automatically under the assumption that we have exponential decay of correlations in BV norm  versus
$\mathscr{L}^1$ and
the transfer operator $P$ is defined with respect to the invariant measure $\mu$
in the usual way by $\int (P f) g\,d\mu=\int fg(T) \,d\mu$ for all $f\in \mathscr{L}^1 (\mu)$,
$g\in \mathscr{L}^{\infty}(\mu)$. Hence $P1=1$ and in particular $|P\phi|_\infty\le|\phi|_\infty$.
 We write $P^n$ for the $n$-fold
composition of  the linear operator $P$. Let $\tilde{\phi_i}=\phi_i -\mu (\phi_i)$.
As before define
$h_n=\sum_{j=1}^nP^j\tilde{\phi}_{n-j}$ and write
\[
\psi_n=\tilde{\phi}_n + h_n - h_{n+1}\circ T.
\]
 Again, for convenience we put
 \[
 U_n= \psi_n\circ T^n
 \]
so that $(U_n)$ is a sequence of reversed martingale differences for the filtration $(\mathcal{B}_n)$.
As in the case of sequential expanding maps one shows that
 $\sum_{i=1}^n E[U_i^2]=\sigma_n^2+O(\sigma_n)$.  Condition~(A)
 of Theorem~\ref{Cuny_Merlevede} holds exactly as before.

In order to estimate $\mu (|U_n|^4)$ observe that by Minkovski's inequality ($p>1$)
$$
\|h_n\|_{p}\le \sum_{j=1}^{n-1}\|P^j\tilde\phi_{n-j}\|_{p},
$$
where
$$
\|P^j\tilde\phi_{n-j}\|_{p}\le \|P^j\tilde\phi_{n-j}\|_{BV}
\le c_1\vartheta^j\|\tilde\phi_{n-1}\|_{BV}\le c_2\vartheta^j
$$
for all $n$ and $j<n$. For small values of $j$ we use the estimate (as $|\tilde\phi_{n-j}|_\infty\le1$)
$$
\int\left|P^j\tilde\phi_{n-j}\right|^p\le\int\left|P^j\tilde\phi_{n-j}\right|
\le\int P^j(\phi_{n-j}+\mu(A_{n-j}))
=\int\phi_{n-j}\circ T^j+\mu(A_{n-j})
=2\mu(A_{n-j}).
$$
If we let $q_n$ be the smallest integer so that $\vartheta^{q_n}\le(\mu(A_{n-q_n}))^\frac1p$, then
$$
\|h_n\|_{p}
\le\sum_{j=1}^{q_n}\left(2\mu(A_{n-j})\right)^\frac1p+\sum_{j=q_n}^nc_2\vartheta^j
\le c_3q_n\left(\mu(A_{n-q_n})\right)^\frac1p.
$$
A similar estimate applies to $h_{n+1}$.
Note that $q_n\le c_4\log n$ for some constant $c_4$.
Let us put $p=4$; then factoring out yields
$$
\int\psi_n^4=\mathcal{O}(\mu(A_n))+\|h_n-h_{n+1}T\|_{4}^4
=\mathcal{O}(\mu(A_n))+\mathcal{O}(q_{n+1}^4\mu(A_{n-q_n})).
$$
Let $\alpha<1$ (to be determined below) and put $a_n=E_n^\alpha$, where
$E_n=\sum_{j=1}^n\mu(A_j)$. Then
$$
\sum_n\frac{\mu(U_n^4)}{a_n^2}
\le c_5\sum_n\frac{\mu(A_n)+q_{n+1}^4\mu(A_{n-q_n})}{E_n^{2\alpha}}
\le c_6\sum_n\frac{q_{n+1}^4\mu(A_{n-q_n})}{E_{n-q_n}^{2\alpha}}
\le c_7\sum_n\frac{q_{n+q_n+1}^4\mu(A_{n})}{E_{n}^{2\alpha}}.
$$
Since
$$
\frac{E_n^{2\alpha}}{\mu(A_n)}\ge\left(\sum_{j=1}^n(\mu(A_j)^\frac1{2\alpha}\right)^{2\alpha}
\ge\left(\sum_{j=1}^nj^{-\frac{\gamma}{2\alpha}}\right)^{2\alpha}\ge c_8n^{2\alpha-\gamma}
$$
we obtain the majorisations
$$
\sum_n\frac{\mu(U_n^4)}{a_n^2}
\le\sum_nq_{n+q_n+1}^4n^{\gamma-2\alpha}
\le c_9\sum_nn^{\gamma-2\alpha}\log^4n
$$
which converge if $\alpha>\frac{1+\gamma}2$.
We have thus verified Condition~(B) of Theorem~\ref{Cuny_Merlevede} with the
value $v=2$.

 Thus $U_n$ satisfies the ASIP with error term $o(E_n^\frac{1-\beta}2)=o(\sigma_n^{1-\beta})$
  for any $\beta<\frac{1-\gamma}2$

  Finally
  \[
  \sum_{j=1}^n U_j =\sum_{j=1}^n \tilde{\phi}_j (T^j) + h_1(T_1)- h_n (T^n)
  \]
  and as $ |h_n|$ is uniformly bounded  we conclude that $(\phi_j (T^j))$ satisfies  the ASIP with error term $o(\sigma_n^{1-\beta})$ for all  $\beta<\frac{1-\gamma}2$.
\end{proof}

\begin{Remark}
 We are unable with the present proof to obtain an ASIP in the case $\mu (A_n)=\frac{1}{n}$
  ($\gamma=1$) though a CLT has been proven~\cite{HNVZ,CR}.
 \end{Remark}

\section{ASIP for non-stationary observations on  invertible hyperbolic systems.}

In this section we will suppose that $B_{\alpha}$ is the Banach space of $\alpha$-H\"{o}lder functions on
 a compact metric space $X$ and  that $(T,X,\mu)$
is an ergodic measure preserving transformation. Suppose that $P$ is the $\mathscr{L}^2$ adjoint of the 
Koopman operator $U$, $U\phi=\phi\circ T$, with respect to $\mu$. First we consider  the non-invertible 
case and suppose 
that $\|P^n \phi \|_{\alpha} \le C\vartheta^n \|\phi\|_{\alpha}$ for all $\alpha$-H\"{o}lder $\phi$ such that 
$\int \phi \,d\mu=0$ where $C>0$ and $0<\vartheta<1$ are uniform constants. Under  this assumption 
we will establish 
 the ASIP for 
sequences of uniformly  H\"older functions satisfying a certain variance growth condition.
Then we will give a corollary which establishes  the ASIP for 
sequences of uniformly  H\"older functions on an Axiom A system satisfying the same  variance growth condition.

  The main difficulty in this setting is establishing a strong law of large numbers with 
error (Condition (A)) for  the squares  $(U_j^2)$ of the martingale difference scheme. We are not able to use  the Gal-Koksma lemma in the same way as we did in the setting of decay in bounded variation norm. Nevertheless our results, while clearly not optimal, point the way to establishing strong statistical properties for non-stationary time series of observations on hyperbolic systems.

\begin{Theorem}\label{NSHYP}
Suppose $\{\phi_j\}$ is a sequence of $\alpha$-H\"{o}lder functions such that $\int \phi_j \,d\mu=0$  and $\sup_j\|\phi_j\|_{\alpha} \le C_1$ for some constant $C_1<\infty$.

Let $\sigma_n^2=\int (\sum_{j=1}^n \phi_j\circ T^j )^2d\mu$ and suppose that $\sigma_n^2 \ge C_2n^{\delta}$ for some $\delta>\frac{\sqrt{17}-1}4$ and a constant $C_2<\infty$.  
Then there is a sequence of centered  independent Gaussian random variables
$(Z_j)$  such that, enlarging our probability space if necessary, 
\[
\sum_{j=1}^n \phi_j\circ T^j =\sum_{j=1}^n Z_j +\mathcal{O}(\sigma_n^{1-\beta})
\]
$\mu$ almost surely for any $\beta<\frac{\sqrt{17}-1}{4\delta}$.

Furthermore $ \sum_{i=1}^n E[Z_i^2]=\sigma_n^2+\mathcal{O}(\sigma_n)$.

\end{Theorem}

\begin{proof}
Define $h_n=P\phi_{n-1}+P^{2}\phi_{n-2}+\cdots+P^n\phi_0$
and put
\[
\psi_n=\phi_n +h_n-h_{n+1}\circ T.
\]
Note $P\psi_n=0$ and that $\|h_n\|=\mathcal{O}(1)$ for $n>1$ by the same argument as in the proof 
of Theorem~\ref{ASIP.sequential}.
The sequence $U_n=\psi_n \circ T^n$ is a sequence of reversed martingale differences with respect to the filtration $\mathcal{F}_n$,
where $\mathcal{F}_n=T^{-n}\mathcal{F}_0$. 
We will take $a_n=\sigma_n^{2\eta}$ where $\eta>0$
will be determined below.
Since $ \|\psi_j\|_{\alpha}=\mathcal{O}(1)$ 
and consequently $ \|U_j\|_{\alpha}=\mathcal{O}(1)$ we conclude that
$$
\sum_n\frac{\mu(U_n^4)}{a_n^2}\le c_1\sum_n\frac1{\sigma_n^{4\eta}}
\le c_2\sum_n\frac1{n^{2\eta\delta}}
<\infty
$$
provided $\eta>\frac1{2\delta}$. In this case Condition~(B) of  Theorem~\ref{Cuny_Merlevede} is satisfied
for $v=2$. 

In order to verify  Condition~(A) of Theorem~\ref{Cuny_Merlevede} let us observe that
  $E[U_j^2|\mathcal{F}_{j+1}]=E[\psi_j^2\circ T^j |\mathcal{F}_{j+1}]= P^{j+1}(\psi_j\circ T^{j})\circ T^{j+1}=(P^{j+1}U^{j}\psi_j^2)\circ T^{j+1}=(P\psi_j^2)\circ T^{+1}$. We now shall prove a strong law of large 
  numbers with rate for the sequence $E[U_j^2|\mathcal{F}_{j+1}]$. For simplicity of notation we denote  $E[U_j^2|\mathcal{F}_{j+1}]$ by $\hat{U_j}^2$.

Let us write $S_n= \sum_{j=1}^n [\hat{U_j}^2 -\mu (U^2_j)]$ for the LHS of condition~(A) in 
Theorem~\ref{Cuny_Merlevede}.
Then $\rho_n^2=\int S_n^2\,d\mu=\int (\sum_{j=1}^n \hat{U_j}^2 -E[U_j^2])^2 \,d\mu$
satisfies by decay of correlations the estimate $\rho_n^2=\mathcal{O}(n)$,
where we used that $ \|\hat{U_j}^2\|_{\alpha}=\mathcal{O}(1)$.
Hence by Chebyshev's inequality
$$
P\!\left(|S_n|> \frac{\sigma_n^{2\eta}}{\log n}\right) \le \frac{\rho_n^2}{\sigma_n^{4\eta}}\log^2 n
\le c_3n^{-(2\eta\delta-1)}\log^{2}n
$$
as $\sigma_n^2=\mathcal{O}(n^\delta)$.
Since $\delta$ is never larger than $2$, we have  $2\eta\delta-1\le1$. Then along a subsequence $f(n)=[n^\omega]$ for $\omega>\omega_0=\frac1{2\eta\delta-1}\ge1$ we can apply the 
Borel-Cantelli lemma since
 $P\!\left(|S_{f(n)}|>\sigma_{f(n)}^{2\eta}/\log f(n)\right)$ is summable as 
 $\sum_n n^{-\omega(2\eta\delta-1)}\log^{2} n<\infty$. 
Hence by Borel-Cantelli for $\mu$ a.e. $x \in X$, $|S_{f(n)}(x)|>\frac{\sigma_{f(n)}^{2\eta}}{\log f(n)}$ only finitely often.

In order to control the gaps note that  $[(n+1)^\omega] -[n^\omega]=\mathcal{O}(n^{\omega-1})$
and let $k\in(f(n),f(n+1))$.
Since along the subsequence $S_{f(n)} =o(\sigma_{f(n)}^{2\eta})$ we conclude that
$S_k=o(\sigma^{2\eta}_{f(n)} )  +\mathcal{O}(n^{\omega-1})$ as there are at most 
$n^{\omega-1}$ terms  $\hat{U_j}^2-E[U_j^2]=\mathcal{O}(1)$ in the range $j\in(f(n),k]$.
 
Choosing $\omega>\omega_0 $ close enough to $\omega_0$ we conclude that
 $$
 S_k=o\!\left(\sigma_{f(n)}^{2\eta}+n^{\omega-1}\right)
 =o\!\left(\sigma_n^{2\eta}+\sigma_n^{(\omega-1)\frac2\delta}\right)
 =o\!\left(\sigma_k^{2\eta}\right),
 $$
 for $\eta>\eta_0$ where 
 $\eta_0$ satisfies $2\eta_0=(\omega_0-1)\frac2\delta=\frac{2-2\eta\delta}{2\eta\delta-1}\frac2\delta$ 
 which implies
$ \eta_0=\frac{\gamma_0}{\delta}$,  with $\gamma_0=\frac{\sqrt{17}-1}4$.

This concludes the proof of Condition~(A) with $a_n=\sigma_n^{2\eta}$. 
Also note that $\eta_0$ is larger than $\frac1{2\delta}$ which ensures Condition~(B). 
Thus $\{U_j\}$ satisfies  the ASIP
with error $\mathcal{O}(\sigma_n^{1-\beta})$ for $0<\beta<\beta_0= 1-\eta_0=1-\frac{\gamma_0}\delta$
and hence so does $\{\phi_j\circ T^j\}$. In particular we must require $\delta$ to be bigger than
$\gamma_0$ (which is slightly larger than $\frac34$).
\end{proof}



We now state a  corollary of this theorem for a sequence of  non-stationary observations on  
 Axiom A dynamical systems. 

\begin{cor}\label{axiomAmain}
Suppose $(T,X,\mu)$ is an Axiom-A dynamical system, where $\mu$ is a Gibbs measure. Suppose $\{\phi_j\}$ is a sequence of $\alpha$-H\"{o}lder functions such that $\int \phi_j\, d\mu=0$ and $\sup_j \|\phi_j\|_{\alpha} <\infty$ for some constant $C$. 
Let $\sigma_n^2=\int (\sum_{j=1}^n \phi_j \circ T^n)^2d\mu$ and suppose that $\sigma_n^2 \ge Cn^{\delta}$ for some $\delta>\frac{\sqrt{17}-1}4$ and a constant $C<\infty$.  Then there is a sequence of centered  independent Gaussian random variables
$(Z_j)$ and a $\gamma >0$ such that, enlarging our probability space if necessary, 
\[
\sum_{j=1}^n \phi_j \circ T^j =\sum_{j=1}^n Z_j +O(\sigma_n^{1-\beta})
\]
$\mu$ almost surely for any $\beta<\frac{\sqrt{17}-1}{4\delta}$.

Furthermore $ \sum_{i=1}^n E[Z_i^2]=\sigma_n^2+O(\sigma_n)$.
\end{cor}

\begin{proof}
The  assumption $\sigma_n^2  \ge C n^{\delta}$ for some $\delta>\frac{\sqrt{17}-1}4$ agrees with 
Theorem~\ref{NSHYP}.
The basic strategy is now  the standard technique of coding  first by a two sided shift and then reducing to a non-invertible  one-sided shift. There is a good description  in Field, Melbourne and T\"or\"ok~\cite{FMT}.
We use a Markov partition to code $(T,X,\mu)$ by a 2-sided shift $(\sigma,\Omega,\nu)$ in a standard way~\cite{Bowen,Parry_Pollicott}.  We lift $\phi_j$
to the system $(\sigma,\Omega,\nu)$ keeping the same notation for $\phi_j$ for simplicity. Using  the Sinai trick~\cite[Appendix A]{FMT}
we may write
\[
\phi_j=\psi_j+ v_j-v_{j+1}\circ \sigma
\]
where $\psi_j$ depends only on future coordinates and is H\"older of exponent $\sqrt{\alpha}$ if $\phi_j$ is of exponent $\alpha$.  In fact
$\|\psi_j\|_{\sqrt{\alpha}}\le K $ and similarly $\|v_j\|_{\sqrt{\alpha}}\le K $ for a  uniform constant $K$.

There is a slight difference in this setting to the usual construction. Pick a H\"older map $G:X\to X$ that depends only on future coordinates (e.g.\ a map which locally substitutes all negative coordinates by a fixed string) and define
\[
v_n(x)=\sum_{k\ge n} \phi_k (\sigma^{k-n}x)-\phi_k (\sigma^{k-n} Gx).
\]
It is easy to see that the sum converges  since $|\phi_k (\sigma^{k-n}x)-\phi_k (\sigma^{k-n} Gx)|\le C\lambda^k \|\phi_k\|_{\alpha}$ (where $0<\lambda<1$)
 and that 
$\|v_n\|_{\alpha}\le C_2$ for some uniform $C_2$.

Since
\[
\phi_n -v_n+v_{n+1}\circ \sigma =\phi_n (Gx)+ \sum_{k>n}[\phi_{k} (\sigma^{k-n} Gx)- \phi_{k} (\sigma^{k-n} G\sigma x) ]
\]
defining $\psi_n=\phi_n -v_n+v_{n+1}\circ \sigma $ we see $\psi_n$ depends only on future coordinates.

We let $\mathcal{F}_0$ denote the $\sigma$-algebra
consisting of events which depend on past coordinates. This is equivalent to conditioning on local stable manifolds defined by the Markov partition.  Symbolically
$\mathcal{F}_0$ sets are of the form $(****.\omega_0 \omega_1 \ldots)$ where $*$ is allowed to be any symbol.

Finally using the transfer operator  $P$ associated to the one-sided shift $\sigma (x_0x_1\ldots x_n \ldots )=(x_1x_2\ldots x_n \ldots )$ we are in the set-up of Theorem~\ref{NSHYP}.
As before we define $h_n=P\psi_{n-1}+P^{2}\psi_{n-2}+\cdots+P^n\psi_0$
and put
\[
V_n=\psi_n +h_n-h_{n+1}\circ T
\]
The sequence $U_n=V_n \circ T^n$ is a sequence of reversed martingale differences with respect to the filtration $\mathcal{F}_n$,
where $\mathcal{F}_n=\sigma^{-n}\mathcal{F}_0$. In fact $(UP)f=E[f|\sigma^{-1}\mathcal{F}_0]\circ \sigma$ while $(PU) f=f$ (this is easily checked, see~\cite[Remark 3.1.2]{FMT}  or~\cite{Parry_Pollicott}).

 Thus $U_n$ satisfies the ASIP with error term $o(\sigma_n^{1-\beta})$ for  
 $\beta\in(0,1-\frac{\gamma_0}\delta)$. Hence $\psi_n\circ T^n $ satisfies the ASIP
 with error term $o(\sigma_n^{1-\beta})$.

  Finally
  \[
  \sum_{j=0}^n \phi_j =\sum_{j=0}^n  \psi_j (T^j) + [v_0-v_n\circ \sigma^{n+1}]
  \]
  as the sum telescopes.
  As $ |v_n|\le C$  we have the  ASIP with error term $o(\sigma_n^{1-\beta})$ for  the sequence $\{\phi_n\circ T^n\}$.   This concludes the proof.
  \end{proof}




\section{Improvements of  earlier work.}

We collect here examples for which a self-norming CLT was already proven,
but actually a (self-norming) ASIP holds if the variance grows at the rate
required by Theorem~\ref{Main_BV}.

Conze and Raugi~\cite[Remark 5.2]{CR} show that for sequential systems
formed by taking maps near a given $\beta$-transformation with $\beta>1$,
by which we mean maps $T_{\beta{'}}$ with $\beta{' }\in (\beta -\delta,
\beta +\delta)$ for sufficiently small $\delta>0$, the conditions~(DFLY)
and~(LB) are satisfied and if $\phi$ is not a coboundary
for $T_{\beta}$ then the variance for $\phi\in BV$ grows as $\sqrt{n}$.

N\'andori, Sz\'asz and Varj\'u~\cite[Theorem 1]{NSV} give conditions under
which sequential systems satisfy a self-norming CLT. These conditions
include~(DFLY) and~(LB) (the maps all preserve a fixed measure
$\mu$, so one can use the transfer operator with respect to $\mu$), and their
main condition gives the rate of growth for the variance (see~\cite[page
1220]{NSV}). If this rate satisfies the requirement of
Theorem~\ref{Main_BV}, then for such systems the ASIP holds as well. Such
cases follow from their Examples 1 and 2, where the maps are selected from
the family $T_a(x)=ax (\operatorname{mod} 1)$, $a\ge 2$ integer, and
Lebesgue as the invariant measure.
Note however that their Example 2 includes sequential systems whose
variance growth slower than any power of $n$, but still satisfy the
self-norming CLT.


%

\section{Further applications.}


%
We consider here maps for which conditions~(DFLY) and~(LB)
are satisfied, but in order to guarantee the unboundedness of the variance
when $\phi$ is not a coboundary, we need to introduce new assumptions; we
follow here again \cite{CR}, especially Sect. 5. First of all, all the maps
in $\cal{F}$ will be close, in a sense we will describe below, to a given
map $T_0$. Call $P_0$ the transfer operator associated to $T_0.$
Then one considers  the following distance between two operators $P$ and $Q$
acting
on $BV$:
$$
d(P,Q)= \sup_{f\in BV,\; \|f\|_{BV}\le 1}||P f-Q f||_1.
$$ By
induction and the Doeblin-Fortet-Lasota-Yorke inequality for compositions we immediately
have
\begin{equation}\label{VC}
\textbf{(DS)\qquad}  d(P_r \circ \cdots \circ P_1, P_0^r)\le M \sum_{j=1}^r d(P_j, P_0),
\end{equation}
with $M=1+A\rho^{-1}+B.$

{\bf Exactness property}: The operator $P_0$ has a spectral gap, which
implies that there are two constants $C_1<\infty$ and $\gamma_0\in(0,1)$
so that
$$
\textbf{(Exa)\qquad} ||P_0^n f||_{BV}\le C_1 \gamma_0^n||f||_{BV}
$$
for all $f\in BV$ of zero (Lebesgue) mean and $n\ge 1$.

According to~\cite[Lemma~2.13]{CR}, (DS) and~(Exa) imply that there exists a constant
$C_2$ such that
$$
\|P_n\circ\cdots\circ P_1\phi-P_0^n\phi\|_{1}\le C_2\|\phi\|_{BV}(\sum_{k=1}^pd(P_{n-k+1},P_0)+(1-\gamma_0)^{-1}\gamma_0^p)
$$
 for all
integers $p\le n$ and all functions $\phi$ of bounded variation.

{\bf  Lipschitz continuity property}: Assume that the maps (and their transfer
operators) are parametrized by a sequence of numbers $\eps_k$, $k\in
\mathbb{N}$, such that $\lim_{k\to\infty}\eps_k=\eps_0$, ($P_{\eps_0}=P_0$).
We assume that there exists a constant $C_3$ so that
$$
\textbf{(Lip)\qquad} d(P_{\eps_k},P_{\eps_j})\le C_3|\eps_k-\eps_j|,\qquad \text{ for all $k, j\ge 0$}.
$$

{\bf Convergence property}: We require algebraic convergence of the
parameters, that is, there exist a constant $C_4$ and $\kappa>0$ so that
$$
\textbf{ (Conv)\qquad} |\eps_n-\eps_0|\le \frac{C_4}{n^{\kappa}}\qquad \forall n\ge1.
$$

With this last assumption and~(Lip), we get a polynomial decay for (\ref{VC}) of the
type $O(n^{-\kappa})$ and in particular we obtain the same algebraic
convergence in $\mathscr{L}^1$ of $P_n\circ\cdots\circ P_1\phi$ to $h \int \phi \,dm$,
where $h$ is the density of the absolutely continuous mixing measure of the
map $T_0.$ This convergence is necessary to establish the growth of the
variance $\sigma^2_n$.

Finally, we also require

{\bf Positivity property}: The density $h$ for the limiting map $T_0$ is strictly positive, namely
\[
\textbf{(Pos)\qquad} \inf_{x}h(x)>0.
\]

The relevance of these four properties is summarised by the following result:

\begin{Lemma}\cite[Lemma~5.7]{CR}\label{thm.xtra-conditions}
  Assume the assumptions~(Exa), (Lip), (Conv) and~(Pos) are satisfied.
 If $\phi$ is not a coboundary for $T_0$ then $\sigma^2_n/n$
  converges as $n\to\infty$ to $\sigma^2$ which moreover is given by
  $$
  \sigma^2= \int \hat{P}[G\phi-\hat{P}G\phi]^2(x) h(x) \ dx,
  $$
  where $\hat{P}\phi=\frac{P_0(h\phi)}{h}$ is the normalized transfer operator
  of $T_0$ and $G\phi=\sum_{k\ge0}\frac{P_0^k(h\phi)}{h}.$
\end{Lemma}

\subsection{$\beta$ transformations}
Let $\beta>1$ and denote by $T_{\beta}(x)=\beta x$ mod $1$ the
$\beta$-transformation on the unit circle. Similarly for $\beta_k\ge 1+c > 1$, $ k=1,2,\dots$,
we have the transformations $T_{\beta_k}$ of the same kind, $x\mapsto \beta_kx$ mod $1$.
Then ${\cal F}=\{T_{\beta_k}:k\}$ is the family of functions we want to consider here.
The property~(DFLY) was proved in~\cite[Theorem 3.4~(c)]{CR}
and condition~(LB) in \cite[Proposition 4.3]{CR}.
Namely, for any $\beta > 1$ there exist $a>0, \delta>0$ such that whenever
$\beta_k\in[\beta-a, \beta+a],$ then $P_k\circ \cdots\circ P_1\ 1(x)\ge
\delta,$ where $P_\ell$ is the transfer operator of $T_{\beta_\ell}.$
%
The invariant density of $T_{\beta}$ is bounded below, and continuity~(Lip) is
precisely the content of Sect.~5 in \cite{CR}. We therefore obtain
(see~\cite[Corollary 5.4]{CR}):

\begin{Theorem}
  Assume that $|\beta_n-\beta|\le n^{-\theta}$, $\theta > 1/2$. Let $\phi\in BV$
  be such that $m(hf)=0$, where $m$  is the Lebesgue measure
   and $\phi$ is not a coboundary for $T_\beta$, so $\sigma^2\not=0$. Then
   the random variables
\[
W_n=\phi+ T_{\beta_1} \phi + \dots + T_{\beta_1} T_{\beta_2} \dots T_{\beta_{n-1}}\phi
\]
satisfy a standard ASIP with variance $\sigma^2$.
\end{Theorem}

\subsection{Perturbed expanding maps of the circle.}

We consider a $C^2$ expanding map $T$ of the circle
$\mathbb{T}$; let us put $A_k=[v_k, v_{k+1}]; k=1,\cdots,m, v_{m+1}=v_1$
the closed intervals such that $TA_k=\mathbb{T}$ and $T$ is injective over $[v_k,
v_{k+1}).$ The family $\F$ then consists of the perturbed maps $T_{\eps}$ which are
 given by the translations ({\em additive noise}): $T_{\eps}(x)=T(x)+\eps, \ \mbox {mod}\ 1$,
 where
$\eps \in (-1,1).$ We observe that the intervals of local injectivity
$[v_k, v_{k+1}), \ k=1,\cdots,m,$ of $T_{\eps}$ are independent of $\eps$. We
call ${\cal A}$ the partition $\{A_k:k\}$ into intervals of monotonicity. We assume
 there exist constants $\Lambda>1$ and $C_1<\infty$ so that
\begin{equation}\label{du2}
  \inf_{x\in   \mathbb{T}} |DT(x)|\ge \Lambda; \quad
   \sup_{\eps\in(-1,1)}\sup_{x\in \mathbb{T}}\left|\frac{D^2T_{\eps}(x)}{DT_{\eps}(x)}\right|\leq C_1.
\end{equation}

\begin{Lemma} The maps $\F=\{T_\eps:\;\;|\eps|<1\}$ satisfy the conditions of
Lemma~\ref{thm.xtra-conditions}.
\end{Lemma}

\begin{proof} (I) (DFLY) It is well known that any such map
 $T_{\eps}$ satisfying~\eqref{du2} verifies a
 Doeblin-Fortet-Lasota-Yorke inequality $||P_{\eps}f||_{BV}\le \rho ||f||_{BV}+B||f||_1$
 where $\rho\in(0,1)$ and $B<\infty$ are independent of $\eps$
($P_{\eps}$ is the associated transfer operator of $T_\eps$).  For any concatenation
of maps one consequently has
$$
\|\P_nf\|_{BV}\le
\rho^k\|f\|_{BV}+\frac{B}{1-\rho}\|f\|_1,
$$
where $\P_n=P_{\eps_k}\circ \cdots\circ P_{\eps_1}$.

\vspace{2mm}

\noindent (II) (LB) In order to obtain the lower bound property~(LB) we have to consider an upper bound
for concatenations of operators. Since each $T_{\eps}$ has $m$ intervals of monotonicity we have
(where $\TF_n=T_{\eps_n}\circ\cdots\circ T_{\eps_1}$ as before)
\begin{equation}\label{Py}
\P_n1(x)=
\sum_{k_n,\cdots,k_1=1}^m\frac{1}{|D\TF_n(T^{-1}_{k_1,\eps_1}\circ\cdots T^{-1}_{k_n,\eps_n}(x))|}\times {\bf 1}_{\TF_nA_{k_1,\cdots,k_n}^{\eps_1,\cdots,\eps_n}}(x)
\end{equation}
where $T^{-1}_{k_l,\eps_l}, k_l\in[1,m],$ denotes the local inverse of $T_{\eps_l}$ restricted to $A_{k_l}$ and
\begin{equation}\label{DM}
A_{k_1,\cdots,k_n}^{\eps_1,\cdots,\eps_n}= T^{-1}_{k_{1},\eps_{1}}\circ \cdots \circ T^{-1}_{k_{n-1},\eps_{n-1}}A_{k_n}  \cap \cdots \cap T^{-1}_{k_1,\eps_1}A_{k_2}\cap A_{k_1}
\end{equation}
is one of the $m^n$ intervals of monotonicity of $\TF_n$. Since those images
satisfy\footnote{This can be proved by induction; for instance for $n=3$ we have $T_{\eps_3}T_{\eps_2}T_{\eps_1}(T^{-1}_{k_{1},\eps_{1}} T^{-1}_{k_{2},\eps_{2}}A_{k_3}  \cap T^{-1}_{k_1,\eps_1}A_{k_2}\cap A_{k_1})=T_{\eps_3}T_{\eps_2}T_{\eps_1}[T^{-1}_{k_1,\eps_1}( T^{-1}_{k_{2},\eps_{2}}A_{k_3}  \cap A_{k_2}\cap T_{\eps_1}A_{k_1})]=T_{\eps_3}T_{\eps_2}(T^{-1}_{k_{2},\eps_{2}}A_{k_3}\cap A_{k_2}\cap T_{\eps_1}A_{k_1})=T_{\eps_3}T_{\eps_2}[T^{-1}_{k_2, \eps_2}(A_{k_3}\cap T_{\eps_2}A_{k_2}\cap T_{\eps_2}T_{\eps_1}A_{k_1})]=T_{\eps_3}(A_{k_3}\cap T_{\eps_2}A_{k_2}\cap T_{\eps_2}T_{\eps_1}A_{k_1}).$}
\begin{equation}\label{cyl}
\TF_nA_{k_1,\cdots,k_n}^{\eps_1,\cdots,\eps_n}=T_{\eps_n}(A_{k_n}\cap T_{\eps_{n-1}}A_{k_{n-1}}\cap \cdots \cap T_{\eps_{n-1}} \circ \cdots \circ T_{\eps_1} A_{k_1})
\end{equation}
and each branch is onto, we have that the inverse image is the full interval.
By the Mean Value Theorem there exists a point $\xi_{k_1,\cdots,k_n}$ in the interior of the
connected interval $A_{k_1,\cdots,k_n}^{\eps_1,\cdots,\eps_n}$ such that
$|D\TF_n(\xi_{k_1,\cdots,k_n})|^{-1}
=|A_{k_1,\cdots,k_n}^{\eps_1,\cdots,\eps_n}|,$
where $|A|$ denotes the length of the connected interval $A$. In order to get distortion estimates,
let us take two points $u, v$ in the closure of $A_{k_1,\cdots,k_n}^{\eps_1,\cdots,\eps_n}$. Then
($\TF_0$ is the identity map)
\begin{eqnarray*}
\left|\frac{D\TF_n(u)}{D\TF_n(v)}\right|
&=&\exp\left(\log|D\TF_n(u)|-\log|D\TF_n(v)|\right)\\
&=&\exp\sum_{j=1}^{n}\left(\log\left|DT_{\eps_{j}}\circ\TF_{j-1}(u)\right|-
\log\left|DT_{\eps_{j}}\circ\TF_{j-1}(v)\right|\right)\\
&=&\exp\sum_{j=1}^{n}
\frac{|D^2T_{\eps_{j}}(\iota_k)|}{|DT_{\eps_{j}}(\iota_j)|}\left|\TF_{j-1}(u)-\TF_{j-1}(v)\right|
\end{eqnarray*}
for some points $\iota_j$ in $\TF_{j-1}A_{k_1,\cdots,k_n}^{\eps_1,\cdots,\eps_n}$.
Using the second bound in (\ref{du2}) and the fact that $ |\TF_{j-1}(u)-\TF_{j-1}(v)|\le \Lambda^{-(j-1)}$
 we finally have
$$
|D\TF_n(u)/D\TF_n(v)|\le e^{\frac{C_1}{1-\Lambda}}
$$
which in turn implies that
$$
\P_n1(x)\ge e^{-\frac{C_1}{1-\Lambda}}
$$
and this independently of any choice of the $\eps_k, k=1,\cdots,n$ and of $n$.

\vspace{2mm}
\noindent (III) The strict positivity condition~(Pos) holds since the map $T$ is Bernoulli and
for such maps it is well known that its invariant densities are uniformly bounded from below away
from zero~\cite{AF}.

\vspace{2mm}
\noindent(IV) The continuity condition~(Lip) follows the
same proof as in the next section and therefore we refer to that.
\end{proof}

We now conclude by Lemma~\ref{thm.xtra-conditions} the following result:

\begin{Theorem}
Let $\F$ be a family of functions as described in this section. Then for any function
$\phi$ which is not a coboundary for $T_\beta$ we have that the random variables
$$
W_n=\sum_{j=0}^{n-1}\phi\circ\TF_j
$$
satisfy a standard ASIP with variance $\sigma^2$.
\end{Theorem}

\subsection{Covering maps: special cases}\label{EM}
\subsubsection{One dimensional maps}\label{one.dimensional.maps}
The next example concerns piecewise uniformly expanding maps $T$ on the
unit interval. The family ${\cal F}$ will consist of maps $T_{\eps},$ which
are constructed with {\em local} additive noise starting from $T$, which in
turn satisfies:
\begin{itemize}
\item (i) $T$ is locally injective on the open intervals $A_k,
  k=1,\dots,m,$ that give a partition  $\mathcal{A}=\{A_k:k\}$ of the unit
  interval $[0,1]=M$ (up to zero measure sets).
\item (ii) $T$ is $C^2$ on each $A_k$ and has a $C^2$ extension to the
  boundaries. Moreover there exist $\Lambda>1, C_1<\infty$, such that
  $\inf_{x\in M} |DT(x)|\ge \Lambda$ and $\sup_{x\in M}\left|\frac{D^2T(x)}{DT(x)}\right|\leq C_1$.

\end{itemize}

At this point we give the construction of the family $\F$ of maps $T_{\eps}$ by
defining them locally on each interval $A_k$. On each interval
$A_k$ we put $T_{\eps}(x)=T(x)+\eps$ where $|\eps|<1$ and we extend by
continuity to the boundaries.  We restrict to values
of $\eps$ so that the image $T_{\eps}(A_k)$ stays in the unit
interval; this we achieve for a given $\eps$ by choosing the sign of $\eps$ so that
 the image of $A_k$ remains in the unit interval; if not we do {\em not} move the map. The sign will consequently
  vary with each interval.

We add now new the new assumption. Assume there exists a set $\mathcal{J}$ so that:
\begin{itemize}
\item (iii) $\mathcal{J}\subset T_{\eps}A_k$ for all $T_{\eps}\in {\cal F}$ and $k=1,\dots,m$.

\item (iv) The map $T$ send $\mathcal{J}$ on $[0,1]$ and therefore it will not be affected there by the addition of $\eps.$ In particular it  will exist $1\ge L'>0$ such that $\forall k=1,\dots,q$ we have $|T(\mathcal{J})\cap A_k|>L'.$
\end{itemize}

\begin{Lemma}
The maps $T_\eps$ satsify the conditions~(DFLY), (LB), (Pos) and (Lip).
\end{Lemma}

\begin{proof} (I) The condition~(DFLY) follows from assumption~(ii).

  \vspace{2mm}
  \noindent (II) In order to prove the lower bound condition~(LB) we begin by observing that,
   thanks to~(iv), the union over the $m^n$ images of the intervals of monotonicity of
  {\em any} concatenation of $n$ maps, still covers $M$.
  Assumption~(iii) above does not require that each branch of the maps in
  ${\cal  F}$ be onto; instead, and thanks again to~(\ref{cyl}), we see that each
  image $\TF_nA_{k_1,\cdots,k_n}^{\eps_1,\cdots,\eps_n}$
  will have at least length $L=\Lambda L'$, so that the reciprocal of the derivative of
  $\TF_n$ over
  $A_{k_1,\cdots,k_n}^{\eps_1,\cdots,\eps_n}$ will be of order
  $L^{-1}|A_{k_1,\cdots,k_n}^{\eps_1,\cdots,\eps_n}|$
  (as before $\TF_n=T_{\eps_n}\circ\cdots\circ T_{\eps_1}$).
   By distortion we
  make it precise by multiplying by the same distortion constant
  $e^{\frac{C_1}{1-\Lambda}}$ as above. In conclusion we have
$$
P_{\eps_n}\circ \cdots\circ P_{\eps_1}1(x)\ge L^{-1} e^{-\frac{C_1}{1-\Lambda}}
$$
(III) To show strict positivity of the invariant density $h$ for the map $T$ we use Assumption~(iv).
Since $h$ is of bounded variation, it will be strictly positive on an open
interval $J$, where $\inf_{x\in J}h(x)\ge h_*$ where $h_*>0$.
We now choose a partition element $R_n$ of the join
${\cal A}^n=\bigvee_{i=0}^{n-1}T^{-i}{\cal A},$  such that $R_n\subset J$.
This is possible by choosing $n$ large enough since the partition ${\cal A}$ is generating.
 By iterating $n$ times forward we achieve that $\TF_nR_n$  covers $\mathcal{J}$ and therefore after $n+1$
iterations the image of $\mathcal{J}$  will cover the entire unit interval. Then for
any $x$ in the unit interval:
$$
h(x)=P^{n+1}h(x)\ge
h(T_{w}^{-(n+1)}(x))\|DT^{n+1}\|_\infty^{-1}
\ge h_* \|DT^{n+1}\|_\infty^{-1},
$$
 where $T_{w}^{-(n+1)}$ is one
of the  inverse branches of $T^{n+1}$ which sends $x$ into $R_n$.

\vspace{2mm}
\noindent (IV) To prove the continuity property~(Lip) we must estimate the difference
$||P_{\eps_1}f-P_{\eps_2}f||_1$ for all $f$ in BV.  We will adapt for that  to the one-dimensional case a similar property proved in the multidimensional setting in Proposition 4.3 in \cite{AFV} We have
\begin{eqnarray*}
 P_{\eps_1}f(x)-P_{\eps_2}f(x)
 &=&E_1(x)+\sum_{l=1}^m( f\cdot{\bf
   1}_{U^c_n})(T^{-1}_{\eps_1,l}x)\left[\frac{1}{|DT_{\eps_1}(T^{-1}_{\eps_1,l}x)|}-\frac{1}{|DT_{\eps_2}(T^{-1}_{\eps_2,l}x)|}\right]+\\
&&+\sum_{l=1}^m\frac{1}{|DT_{\eps_2}(T^{-1}_{\eps_2,l}x)|}[(f\cdot{\bf 1}_{U^c_n})(T^{-1}_{\eps_1,l}x)-(f\cdot{\bf 1}_{U^c_n})(T^{-1}_{\eps_2,l}x)]\\
&=&E_1(x)+E_2(x)+E_3(x)
\end{eqnarray*}
The term $E_1$ comes from those points $x$ which we omitted in the sum because they have only
 one pre-image in each interval of monotonicity.  The total error $E_1=\int E_1(x)\,dx$ is then
 estimated by
 $|E_1|\le  4m|\eps_1-\eps_2|\cdot\|\hat{P}_{\eps}f\|_{\infty}$.
 But $\|\hat{P}_{\eps}f\|_{\infty}\le \|f\|_{\infty}\sum_{l=1}^m\frac{|DT_{\eps_2}(T^{-1}_{\eps_2,l}x')|}{|DT_{\eps_2}
 (T^{-1}_{\eps_2,l}x)|}\frac{1}{|DT_{\eps_2}(T^{-1}_{\eps_2,l}x')|},$
 where  $x'$ is the point  so that $|DT_{\eps_2}(T^{-1}_{\eps_2,l}x')|\cdot|A_l|\ge \eta$,
 and $\eta$ is the minimum of the length $T(A_k), k=1,\dots,m.$ Due to the bounded distortion
 property, the first ratio inside the summation is  bounded by some constant $D_c$; therefore
 $$
 E_1\le 4m|\eps_1-\eps_2|\cdot\|f\|_{\infty} \frac{D_c}{\eta} \sum_{l=1}^m |A_l|
 \le  4m|\eps_1-\eps_2|\cdot\|f\|_{\infty} \frac{D_c}{\eta}
 $$

 We now bound $E_2$. For any $l$, the term in the square bracket  (we
 drop this index in the derivatives in the next formulas), will be equal to
 $\frac{D^2T(\xi)}{[DT(\xi)]^2}|T^{-1}_{\eps_1}(x)-T^{-1}_{\eps_2}(x)|$,
where $\xi$ is an interior point of $A_l.$ The first factor is
 uniformly bounded by $C_1.$ Since
 $x=T_{\eps_1}(T^{-1}_{\eps_1}(x))=T((T^{-1}_{\eps_1}(x))+\eps_1=T((T^{-1}_{\eps_2}(x))+\eps_2=
 T_{\eps_2}(T^{-1}_{\eps_2}(x))$, we obtain
 $|T^{-1}_{\eps_1}(x)-T^{-1}_{\eps_2}(x)|=|\eps_1-\eps_2||DT(\xi')|^{-1},$
 for some $\xi'\in A_l$. We now use distortion to replace $\xi'$ with $T^{-1}_{\eps_1,l}x$ and get
\begin{eqnarray*}
\int |E_2(x)| \,dx&\le& |\eps_1-\eps_2| C_1 D_c
\int \sum_{l=1}^m |f(T^{-1}_{\eps_1,l})|\frac{1}{|DT_{\eps_1}(T^{-1}_{\eps_1,l}x)|}\,dx\\
&=&|\eps_1-\eps_2| C_1 D_c\int P_{\eps_1}(|f|)(x) dx\\
&=&|\eps_1-\eps_2| C_1 D_c\|f\|_1.
\end{eqnarray*}
To bound the third error term we use formula~(3.11) in~\cite{CR}
$$
\int \sup_{|y-x|\le t}|f(y)-f(x)|dx\le 2t\mbox{Var}(f).
$$
and again use the fact that $|T^{-1}_{\eps_1}(x)-T^{-1}_{\eps_2}(x)|=|\eps_1-\eps_2||DT(\xi')|^{-1}$,
for some  $\xi'\in A_l.$ Integrating $E_3(x)$ yields
$$
\int |E_{3}(x)|dx
\le \ 2m \Lambda^{-1} |\eps_1-\eps_2|\mbox{Var}(f{\bf 1}_{U^c_n})
\le10m \Lambda^{-1}\  |\eps_1-\eps_2|\mbox{Var}(f)
$$
Combining the three error estimates we conclude that there exists a constant $\tilde{C}$ such that
$$
||P_{\eps_1}f-P_{\eps_2}f||_1\le \tilde{C} |\eps_1-\eps_2|\|f\|_{BV}.
$$
\end{proof}

\begin{Theorem}
Let $\F$ be the family of maps defined above and consisting of  the sequence $\{T_{\eps_k}\},$ where the sequence $\{\eps_k\}_{k\ge 1}$ satisfies  $|\eps_k|\le k^{-\theta}$, $\theta > 1/2$. If $\phi$ is not a coboundary for $T$, then
$$
W_n=\sum_{j=0}^{n-1}\phi\circ\TF_j
$$
satisfies  a standard  ASIP with variance $\sigma^2$.
\end{Theorem}

\subsubsection{Multidimensional maps}
 We give here a multidimensional version of the maps
considered in the preceding section; these maps were extensively
investigated in \cite{S, HV, AFV, AFLV, HNVZ} and we defer to those papers
for more details. Let $M$ be a compact subset of $\mathbb{R}^N$
which is the closure of its non-empty interior. We take a map $T:M\to M$
and let $\mathcal{A}=\{A_i\}_{i=1}^m$ be a finite family of disjoint open sets such
that the Lebesgue measure of $M\setminus\bigcup_{i}A_i$ is zero, and there
exist open sets $\tilde{A_i}\supset\overline{A_i}$ and $C^{1+\alpha}$ maps
$T_i: \tilde{A_i}\to \mathbb{R}^N$, for some real number $0<\alpha\leq 1$
and some sufficiently small real number $\eps_1>0$ such that
\begin{enumerate}
\item $T_i(\tilde{A_i})\supset B_{\eps_1}(T(A_i))$ for each $i$, where
  $B_{\eps}(V)$ denotes a neighborhood of size $\eps$ of the set $V.$ The
  maps $T_i$ are the local extensions of $T$ to the $\tilde{A_i}.$

\item there exists a constant $C_1$ so that for each $i$ and $x,y\in T(A_i)$ with
$\mbox{dist}(x,y)\leq\eps_1$,
$$
|\det DT_i^{-1}(x)-\det DT_i^{-1}(y)|\leq C_1|\det DT_i^{-1}(x)|\mbox{dist}(x,y)^\alpha;
$$

\item there exists $s=s(T)<1$ such that $\forall x,y\in T(\tilde{A}_i) \textrm{ with } \mbox{dist}(x,y)\leq\eps_1$, we have
$$\mbox{dist}(T_i^{-1}x,T_i^{-1}y)\leq s\, \mbox{dist}(x,y);$$

\item each $\partial A_i$ is a codimension-one embedded compact
  piecewise $C^1$ submanifold and
  \begin{equation}\label{sc}s^\alpha+\frac{4s}{1-s}Z(T)\frac{\gamma_{N-1}}{\gamma_N}<1,\end{equation}
  where $Z(T)=\sup\limits_{x}\sum\limits_i \#\{\textrm{smooth pieces
    intersecting } \partial A_i \textrm{ containing }x\}$ and $\gamma_N$ is
  the volume of the unit ball in $\mathbb{R}^N$.
\end{enumerate}

Given such a map $T$ we define locally on each $A_i$ the map $T_{\eps}$ by
$T_{\eps}(x):=T(x)+\eps$ where now $\eps$ is an $n$-dimensional vector with
all the components of absolute value less than one. As in the previous
example the translation by $\eps$ is allowed if the image $T_{\eps}A_i$
remains in $M$: in this regard, we could play with the sign of the
components of $\eps$ or do not move the map at all. As in the one dimensional case,
we shall also make the following  assumption on ${\cal F}$.
We assume that there exists a set $\mathcal{J}$ satisfying:

\begin{itemize}
\item [(i)] $\mathcal{J}\subset T_{\eps}A_k$ for all $\forall \;T_{\eps}\in{\cal F}$ and
  for all $k=1,\dots,m$.

\item [(ii)]  $T\mathcal{J}$ is the whole $M$, which in turn implies that there exists $1\ge L'>0$ such that $\forall k=1,\dots,q$ and $\forall T_{\eps}\in {\cal F},$ $\mbox{diameter}(T_{\eps}(\mathcal{J})\cap A_k)>L'.$ 
\end{itemize}

As $\mathcal V\subset \mathscr{L}^1(m)$ we use the space of quasi-H\"older functions,
for which we refer again to \cite{S, HV}.

\begin{Theorem}
  Assume $T:M\to M$ is a map as above such that it has only one absolutely
  continuous invariant measure, which is also mixing. If conditions (i) and
  (ii) hold, let $\F$ be the family of maps  consisting of  the sequence $\{T_{\eps_k}\},$ where the sequence $\{\eps_k\}_{k\ge 1}$ satisfies  $||\eps_k||\le k^{-\theta}$, $\theta > 1/2$. If $\phi$ is not a coboundary for $T$, then
$$
W_n=\sum_{j=0}^{n-1}\phi\circ\TF_j
$$
satisfies a standard  ASIP with variance $\sigma^2$.
\end{Theorem}

\begin{proof}
  The transfer operator is suitably defined on the space of quasi-H\"older
  functions, and on this functional space it satisfies a Doeblin-Fortet-Lasota-Yorke
  inequality. The proof of the lower bound condition~(LB) follows the same path
  taken in the one-dimensional case in Section~\ref{one.dimensional.maps}
   using the distortion bound on the determinants
and Assumption~(ii) which ensures that the images of the domains of local injectivity of any
  concatenation have diameter large enough.  The positivity of the density follows by
  the same argument used for maps of the unit interval since the space of
  quasi-H\"older functions has the nice property that a non-identically
  zero function in such a space is strictly positive on some ball~\cite{S}.
  Finally, Lipschitz continuity has been proved for additive noise in Proposition~4.3 in~\cite{AFV}.
  %
  %
\end{proof}

\subsection{Covering maps: a general class}
We now present a more general class of examples which were introduced in~\cite{BV} to study
metastability for randomly perturbed maps. As before the family ${\cal F}$ will be constructed
 around a given map $T$ which is again defined on the unit interval $M$. We therefore begin to introduce such a map $T$. \\
 {\bf (A1)} There exists a partition $\mathcal{A}=\{A_i:i=1,\dots,m\}$ of $M$, which consists
 of pairwise disjoint intervals $A_i$. Let $\bar A_i :=[c_{i,0},c_{i+1,0}]$. We assume there
 exists $\delta>0$  such that $T_{i,0}:= T|_{(c_{i,0},c_{i+1,0})}$ is $C^2$ and extends to a
  $C^2$ function $\bar T_{i,0}$ on a neighbourhood $[c_{i,0}-\delta,c_{i+1,0}+\delta]$ of $\bar A_i $ ;\\
 {\bf (A2)} There exists $\beta_0<\frac12$ so that $\inf_{x\in I\setminus {\cal C}_0}|T'(x)|\ge\beta_0^{-1}$, where
  ${\cal C}_0=\{c_{i,0}\}_{i=1}^{m}$. \\
 \\ We note that Assumption {\bf (A2)}, more precisely the fact that
 $\beta_0^{-1}$ is strictly bigger than $2$ instead of $1$, is sufficient
 to
 get the uniform Doeblin-Fortet-Lasota-Yorke inequality~\eqref{LLYY} below, as explained in Section 4.2 of \cite{GHW}. We now construct the family ${\cal F}$ by choosing maps $T_{\eps}\in {\cal F}$ close to $T_{\eps=0}:=T$ in the following way:\\
 Each map $T_{\eps}\in {\cal F}$ has $m$ branches and there exists a partition of $M$
 into intervals $\{A_{i,\eps}\}_{i=1}^{m}$, $A_{i,\eps}\cap A_{j,\eps}=\emptyset$ for $i\not= j$, $\bar A_{i,\eps} :=[c_{i,\eps},c_{i+1,\eps}]$ such that\\
 \begin{itemize}
 \item [(i)] for each $i$ one has that
   $[c_{i,0}+\delta,c_{i+1,0}-\delta]\subset
   [c_{i,\eps},c_{i+1,\eps}]\subset [c_{i,0}-\delta,c_{i+1,0}+\delta]$;
   whenever $c_{1,0}=0$ or $c_{q+1},0=1$, we {\em do not move} them with
   $\delta$. In this way we have established a one-to-one correspondence
   between the unperturbed and the perturbed extreme points of $A_i$ and
   $A_{i, \eps}$. (The quantity $\delta$ is from Assumption~(A1) above.)

 \item [(ii)] The map $T_{\eps}$ is locally injective over the closed
   intervals $\overline{A_{i,\eps}}$, of class $C^2$ in their interiors,
   and expanding with $\inf_x|T_{\eps}'x|>2$. Moreover there exists
   $\sigma>0$ such that $\forall T_{\eps} \in {\cal F}, \forall i=1,\cdots,
   m$ and $\forall x \in [c_{i,0}-\delta,c_{i+1,0}+\delta]\cap \overline{A_{i,\eps}}$ where $c_{i,0}$
   and $c_{i, \eps}$ are two (left or right) {\em corresponding points} we
   have:
   \begin{equation}\label{C1}
     |c_{i,0}-c_{i, \eps}|\le \sigma
   \end{equation}
   and
   \begin{equation}\label{C2}
     |\bar{T}_{i,0}(x)-T_{i, \eps}(x)|\le \sigma.
   \end{equation}
 \end{itemize}

 Under these assumptions and by taking, with obvious notations, a
 concatenation of $n$ transfer operators, we have the uniform
 Doeblin-Fortet-Lasota-Yorke
 inequality, namely there exist $\eta\in(0,1)$ and $B<\infty$ such that for all $f\in
 BV$, all $n$ and all concatenations of $n$ maps of ${\cal F}$ we have
\begin{equation}\label{LLYY}
||P_{\eps_n}\circ \cdots\circ P_{\eps_1}f||_{BV}\le \eta^n||f||_{BV}+B||f||_1.
\end{equation}
In order to deal with lower bound condition~(LB), we have to restrict the class
of maps just defined. This class was first introduced in an unpublished,
but circulating, version of~\cite{BV}. A similar class has also been used
in the recent paper~\cite{AR}: both are based on the adaptation to the
sequential setting of the covering conditions introduced formerly by Collet~\cite{C}
and then generalized by Liverani~\cite{LC}. In the latter, the
author studied the Perron-Frobenius operator for a large class of uniformly
piecewise expanding maps of the unit interval $M;$ two ingredients are
needed in this setting. The first is that such an operator satifies the
Doeblin-Fortet-Lasota-Yorke inequality on the pair of adapted spaces
$BV \subset \mathscr{L}^1(m).$
The second is that the cone of functions
$$
{\cal G}_a=\{g\in BV; \ g(x)\neq 0; \ g(x)\ge 0, \forall x\in M; \ \mbox{Var}\ g \le a\int_Mg\,dm\}
$$
for $a>0$ is invariant under the action of the operator. By using the
inequality~\eqref{LLYY} with the norm $\|\cdot\|_{BV}$ replaced by the
total variation $\mbox{Var}$ and using the notation (\ref{o}) for the
arbitrary concatenation of $n$ operators associated to $n$ maps in ${\cal
  F}$ we see immediately that
$$
\forall n, \ \overline{P}_n {\cal G}_a \subset {\cal G}_{ua}
$$
with $0<u<1$, provided we choose $a>B (1-\eta)^{-1}.$ The next result from~\cite{LC} is Lemma 3.2 there, which asserts that given a partition, mod-$0$, ${\cal P}$ of $M$, if each element $p\in {\cal P}$ is a connected interval with Lebesgue measure less than $1/2a,$ then for each $g\in  {\cal G}_a$, there exists $p_0\in {\cal P}$ such that $g(x)\ge \frac12 \ \int_M g \,dm,$ $\forall x\in p_0.$ Before continuing we should  stress that contrarily to the interval maps investigated above,  the domain of injectivity are now (slightly) different from map to map, and in fact we used the notation $A_{i,\eps_k}$ to denote the $i$ domain of injectivity of the map $T_{\eps_k}.$ Therefore the sets (\ref{DM}) will be now denoted as
$$
A_{k_1,\cdots,k_n}^{\eps_1,\cdots,\eps_n}= T^{-1}_{k_{1},\eps_{1}}\circ \cdots \circ T^{-1}_{k_{n-1},\eps_{n-1}}A_{k_n,\eps_n}  \cap \cdots \cap T^{-1}_{k_1,\eps_1}A_{k_2,\eps_2}\cap A_{k_1\eps_1}
$$
Since we have supposed that $\inf_{T_{\eps}\in {\cal F}, i=1,\dots,m, x\in A_{i,\eps}}|DT_{\eps}(x)|\ge \beta_0^{-1}>2,$ it follows that the previous intervals have all  lengths bounded by $\beta_0^n$ independently of the concatenation we have chosen. We are now ready to strengthen the assumptions on our maps by requiring the following condition:

\vspace{2mm}

\noindent {\bf Covering Property:} There exist $n_0$ and $N(n_0)$ such that:\\
(i) The partition into sets $A_{k_1,\cdots,k_{n_0}}^{\eps_1,\cdots,\eps_{n_0}}$ has diameter less than $\frac{1}{2au}.$\\
(ii) For any sequence $\eps_1,\dots,\eps_{N(n_0)}$ and $k_1,\dots,k_{n_0}$ we have
$$
T_{\eps_{N(n_0)}}\circ\cdots\circ T_{\eps_{n_0+1}}A_{k_1,\cdots,k_{n_0}}^{\eps_1,\cdots,\eps_{n_0}}=M
$$

\vspace{2mm}

We now consider $g=1$ and note that for any $l,$ $\overline{P}_l1\in {\cal G}_{ua}$. Then for any
 $n\ge N(n_0),$ we have (from now on using the notation (\ref{o}), we mean that the particular sequence of maps used in the concatenation is irrelevant),
 $\overline{P}^n 1=\overline{P}^{N(n_0)} \overline{P}^{n-N(n_0)}1:= \overline{P}^{N(n_0)} \hat{g},$ where $\hat{g}=\overline{P}^{n-N(n_0)}\ 1.$ By looking at the structure of the sequential operators (\ref{Py}), we see that for any $x\in M$ (apart at most finitely many points for a given concatenation, which is irrelevant since what one  really needs is the $\mathscr{L}^{\infty}_m$ norm  in the
 condition~(LB)), there exists a point $y$ in a set of type $A_{k_1,\cdots,k_{n_0}}^{\eps_1,\cdots,\eps_{n_0}}$, where $\hat{g}(y)\ge \frac12\int_m \hat{g}\,dm,$ and such that $T_{\eps_{N(n_0)}}\circ \cdots \circ T_{\eps_1}y=x.$ This immediately implies that
$$
\overline{P}^n1\ge \frac{1}{2 \beta_M^{N(n_0)}}, \quad \forall \;n\ge N(n_0),
$$
which is the desired result together with the obvious bound
$\overline{P}^l 1\ge \frac{m^{N(n_0)}}{\beta_M},$ for $l<N(n_0),$ and where
 $\beta_M=\sup_{T_{\eps}\in {\cal F}}\max |DT_{\eps}|.$ The positivity condition~(Pos) for the
 density will follow again along the line used before, since the covering condition holds
  in particular for the map $T$ itself. About the continuity~(Lip): looking carefully at the proof
  of the continuity for the expanding map of the intervals, one sees that it  extends
   to the actual case if one gets the following bounds:
\begin{equation}\label{BB}
\left.\begin{array}{r}|T^{-1}_{\eps_1}(x)-T^{-1}_{\eps_2}(x)|\\ |DT_{\eps_1}(x)-DT_{\eps_2}(x)|
\end{array}\right\}=O((|\eps_1-\eps_2|)
\end{equation}
where the point $x$ is in the same domain of injectivity of the maps $T_{\eps_1}$ and $T_{\eps_2}$, the comparison of the {\em same} functions and derivative in two {\em different} points being controlled  controlled  by the condition (\ref{C1}).  The bounds~\eqref{BB} follow easily by adding to~\eqref{C1}, \eqref{C2} the further assumptions that $\sigma=O(\eps)$ and requiring a continuity condition for
derivatives like~\eqref{C2} and with $\sigma$ again being of order $\eps$. With these requirement we can finally state the following theorem
\begin{Theorem}
Let $\F$ be the family of maps constructed  above and  consisting of  the sequence $\{T_{\eps_k}\},$ where the sequence $\{\eps_k\}_{k\ge 1}$ satisfies  $|\eps_k|\le k^{-\theta}$, $\theta > 1/2$. If $\phi$ is not a coboundary for $T$, then
$$
W_n=\sum_{j=0}^{n-1}\phi\circ\TF_j
$$
satisfies a standard  ASIP with variance $\sigma^2$.
\end{Theorem}

\medskip\noindent\textbf{Acknowledgement}.
AT was partially supported by the Simons Foundation grant 239583.  MN was partially supported by NSF grant  DMS 1101315.
SV was supported by the ANR- Project {\em Perturbations} and by the PICS (Projet International de Coop\'eration Scientifique), {\em Propri\'et\'es
  statistiques des syst\`emes dynamiques det\'erministes et al\'eatoires},
with the University of Houston, n. PICS05968. NH and SV thank the University of
Houston for the support and kind hospitality during the preparation of this
work.
SV also acknowledges discussions with R. Aimino about the content of this work.


\begin{thebibliography}{10}
\bibitem{AF} R. Adler, L. Flatto, Geodesic flows, interval maps and symbolic dynamics, {Bull. Amer. Math. Soc.}, {\bf 25}, 1992.
    \bibitem {AFLV} J. Alves, J. Freitas, S. Luzzatto, S. Vaienti, From rates of mixing to recurrence times via large deviations,  {\em Advances in Mathematics},  {\bf 228} (2011), 1203--123.

        \bibitem{AFV} H. Aytach, J. Freitas, S. Vaienti, Laws of rare events for deterministic and random dynamical systems, to appear on {\em Trans. Amer Math. Soc.}

\bibitem{AR} R. Aimino, J. Rousseau, Concentration inequalities for sequential dynamical systems of the unit interval, in preparation.

\bibitem{BH} W. Bahsoun, Ch. Bose, Y. Duan, Decay of correlation for random intermittent maps, arXiv:1305.6588

    \bibitem{BV} W. Bahsoun, S. Vaienti, Escape rates formulae and metastability for randomly perturbed maps, {\em Nonlinearity}, {\bf 26} (2013) 1415--1438.

\bibitem{Berend_Bergelson} D. Berend, and V. Bergelson, Ergodic and Mixing Sequences of Transformations, {\em Ergod. Th. \& Dynam. Sys.}, {\bf 4 }  (1984), 353--366.


\bibitem{Bowen} {R.~Bowen}, {Equilibrium states and the ergodic theory of
    {A}nosov diffeomorphisms}, {\em Lecture Notes in Mathematics}, Vol.
  470, {1975}.


\bibitem{Chernov_Kleinbock} N.~Chernov and D.~Kleinbock, {Dynamical
    {B}orel-{C}antelli lemmas for {G}ibbs measures}, {\em Israel J. Math.},
  {\bf 122} (2001), {1--27}.


\bibitem{C} P. Collet, An estimate of the decay of correlations for mixing non Markov expanding maps of the interval, preprint (1984).

\bibitem{CR} J.-P. Conze, A. Raugi, Limit theorems for sequential expanding dynamical systems on $[0, 1]$, {\em Ergodic
theory and related fields}, {\bf 89121}, Contemp. Math., 430, Amer. Math.
Soc., Providence, RI, 2007

\bibitem{Cuny_Merlevede} C.~Cuny and F.~Merlev\`ede, Strong invariance principles with rate for ``reverse'' martingales and applications. {\it Preprint 2014}.

\bibitem{FMT}  M. Field, I. Melbourne, A. T\"{o}r\"{o}k,
Decay of Correlations, Central Limit Theorems and approximation by Brownian Motion for compact Lie group extensions.
{\em Ergodic Theory and Dynamical Systems}, {\bf 23} (2003), 87--110

\bibitem{Gordin} M.~I.~Gordin, The central limit theorem for stationary processes {Soviet. Math. Dokl.}
{\bf 10} (1969), 1174--1176.

 \bibitem{GO} S. Gou\"{e}zel,  Central limit theorem and stable laws for intermittent maps {\em Probab. Theory Relat. Fields}, {\bf 128}  (2004), 82--122,

\bibitem{Gouezel} S. Gou\"{e}zel,  Almost sure invariance principle for dynamical systems by spectral methods. {\em Ann. Probab.}, {\bf 38(4)} (2010), 1639--1671.

\bibitem{GHW} Gonzalez Tokman, C., Hunt, B. R. and Wright, P.: Approximating invariant densities of metastable systems,  {\em Ergodic Theory Dynam. Systems}, {\bf 31(5)} (2011), 1345--1361.

\bibitem{GOT}C. Gupta, W. Ott, A. T\"or\"ok, Memory loss for time-dependent piecewise expanding systems in higher dimension, {\em Mathematical Research Letters}, {\bf 20}  (2013), 155--175

\bibitem{HH} H Hu, Decay of correlations for piecewise smooth maps with indifferent fixed points, {\em Ergodic Theory  and Dynamical Systems}, {\bf 24} (2004), 495--524

    \bibitem{HV} H. Hu, S. Vaienti, Absolutely continuous invariant measures for non-uniformly expanding maps,  {\em Ergodic Theory and Dynamical Systems},  {\bf 29} (2009), 1185--1215.


\bibitem{HNVZ} N. Haydn, M. Nicol, S. Vaienti and L. Zhang. Central limit theorems for the shrinking target problem, to appear on {\em Journal of Statistical Physics.}

    \bibitem{LC} C. Liverani, Decay of Correlations in Piecewise Expanding maps, {\em Journal of Statistical Physics}, {\bf 78} (1995), 1111--1129.

\bibitem{LSV} C. Liverani, B. Saussol, S. Vaienti, A probabilistic approach to intermittency,   \textit{Ergodic theory and dynamical systems},  {\bf 19} (1999), 671--685.

\bibitem{MN1} I.~Melbourne and M.~Nicol, Almost sure invariance principle for nonuniformly hyperbolic systems. {\em Commun. Math. Phys.}, {\bf 260} (2005), 131--146.

\bibitem{MN2}  I.~Melbourne and M.~Nicol, A vector-valued almost sure invariance principle for hyperbolic dynamical systems.  {\em Ann. Probab.}, (2009), 478--505.

\bibitem{Merlevede_Rio} Merlev\`ede, F. and Rio, E. Strong approximation of partial sums under dependence conditions with application to dynamical systems.
{\em  Stochastic Process. Appl. } {\bf 122} (2012), 386--417.

\bibitem{NSV} P. N\'andori, D. Sz\'asz and Y. Varj\'u. A central limit theorem for time-dependent dynamical systems. {\bf Vol 22(1)}  (2006).

\bibitem{OSY} W. Ott, M. Stenlund, L.-S. Young, Memory loss for time-dependent dynamical systems, {\em  Math. Res. Lett.}, {\bf 16} (2009), 463--475.

\bibitem{Parry_Pollicott} W.~Parry and M.~Pollicott. Zeta functions and the periodic orbit structure of hyperbolic dynamics.  {\em Ast\'erisque} {\bf 187--188} Soci\'et\'e Math\'ematique
de France, Montrouge,  1990.

\bibitem{Philipp_Stout}  W. Philipp and W. Stout ``Almost sure invariance principles for partial sums of weakly dependent random vectors'', Memoirs AMS 1975.




		
\bibitem{Sarig} O. Sarig,  Subexponential decay of correlations, {\em Invent. Math.},  {\bf 150} (2002), 629--653.

\bibitem{SS}W. Shen, S. Van Strien,  On stochastic stability of expanding circle maps with neutral fixed points, {\em Dynamical Systems, An International Journal},  {\bf 28}, (2013)

\bibitem{Sprindzuk} Vladimir G. Sprindzuk, {\em Metric theory of
  Diophantine approximations}, V. H. Winston and Sons, Washington,
  D.C., 1979, Translated from the Russian and edited by Richard
  A. Silverman, With a foreword by Donald J. Newman, Scripta Series in
  Mathematics. MR MR548467 (80k:10048).

\bibitem{S} M. Stenlund, Non-stationary compositions of Anosov diffeomorphisms, {\em Nonlinearity} {\bf 24} (2011) 2991--3018

\bibitem{SYZ} M. Stenlund, L-S. Young, H. Zhang, Dispersing billiards with moving scatterers, {\em Comm. Math. Phys.}, to appear (2013)

\bibitem{Viana} M.~Viana, Stochastic dynamics of deterministic systems, Brazillian Math. Colloquium 1997, IMPA.


    \end{thebibliography}
\end{document}